\documentclass[12pt, reqno]{amsart}

\usepackage{amsmath,amsfonts,amsthm,amssymb,color,hyperref}

\usepackage{mathrsfs}

\usepackage[T1]{fontenc}

\usepackage{graphicx}
\usepackage{subfigure} 

\makeatletter
     \def\section{\@startsection{section}{1}%
     \z@{.7\linespacing\@plus\linespacing}{.5\linespacing}%
     {\bfseries
     \centering
     }}
     \def\@secnumfont{\bfseries}
     \makeatother

\usepackage{graphicx}

\newcommand{\R}{\mathbb R}
\newcommand{\RR}{\mathbb R}

\newcommand{\E}{\mathbb E}

\newcommand{\1}{{\bf 1}}

\newcommand{\HH}{\mathfrak H}

\newtheorem{theorem}{Theorem}[section]
\newtheorem{lemma}[theorem]{Lemma}
\newtheorem{proposition}[theorem]{Proposition}

\theoremstyle{definition}

\theoremstyle{remark}

\numberwithin{equation}{section}
\begin{document}
\title[]{A Central Limit Theorem for the stochastic heat equation}

\author[J. Huang]{Jingyu Huang}
\address{University of Birmingham, School of Mathematics, UK}
\email{j.huang.4@bham.ac.uk}

\author[D. Nualart]{David Nualart} \thanks {D. Nualart is supported by NSF Grant DMS 1811181.}
\address{University of Kansas, Department of Mathematics, USA}
\email{nualart@ku.edu}

\author[L. Viitasaari]{Lauri Viitasaari}
\address{University of Helsinki, Department of Mathematics and Statistics, Finland}
\email{lauri.viitasaari@iki.fi}

\begin{abstract}
We consider the one-dimensional stochastic heat equation driven by a multiplicative space-time white noise. We show that the spatial integral of the solution from $-R$ to $R$ converges in total variance distance to a standard normal distribution as $R$ tends to infinity, after renormalization.   We also show a functional version of this central limit theorem.
\end{abstract}

\maketitle



\allowdisplaybreaks

\section{Introduction}
We consider the one-dimensional stochastic heat equation  
\begin{equation}
\label{eq:heat-equation}
\frac{\partial u}{\partial t} = \frac12 \Delta u + \sigma(u) \dot{W}
\end{equation}
on $\R_+\times \R$, where $\dot{W}$ is a space-time Gaussian white noise, with initial condition $u_0(x) = 1$. 
The coefficient $\sigma$ is a Lipschitz function.   

It is well-known (see, for instance,  \cite{Walsh})
that this equation has a unique  mild solution, which is adapted to the filtration generated by $W$,  such that $\E [|u(t,x)|^2] < \infty$ and it satisfies the evolution equation
\begin{equation}\label{eq: mild}
u(t,x) = 1+ \int_0^t \int_{\RR} p_{t-s}(x-y) \sigma(u(s,y))W(ds,dy)\,, 
\end{equation}
where in the right hand side the stochastic integral is in the sense of Walsh, and $p_t(x)=(2\pi t)^{-1/2}e^{-x^2/(2t)}$ is the heat kernel. 

 In this paper we are interested in the  asymptotic behavior  as $R$ tends to infinity of the quantity 
\begin{equation}
\label{eq:quantity-of-interest}
F_R(t):= \frac{1}{\sigma_R}\left(\int_{-R}^R u(t,x)dx - 2R\right),
\end{equation}
where $R>0$,  $u(t,x)$ is the solution to \eqref{eq:heat-equation} and $\sigma_R^2 = {\rm Var}\left(\int_{-R}^R u(t,x)dx\right)$. 
 
From  equation (\ref{eq: mild}) and the properties of the heat kernel, it follows that the solution to equation (\ref{eq:heat-equation}) satisfies a localization property. This means that, for any fixed $t>0$,  the random variable $u(t,x)$ essentially depends on the noise in a small interval $[x-\epsilon, x+\epsilon]$. This property has been extensively used in the literature, see for example, \cite{CKK, CJK, CJKS}.  

 In particular, for the parabolic Anderson model ($\sigma(u)=u$), it  is shown in \cite{CJK} that for each fixed $t>0$, almost surely, the solution $u(t,x)$ develops high peaks along the $x$-axis. More precisely, it holds that, almost surely
\begin{equation*}
0 < \limsup_{R \to \infty}  \frac{\max_{|x|\leq R} \log u(t,x)}{(\log R)^{2/3}} < \infty\,. 
\end{equation*}
The basic idea in \cite{CJK} to show this result is that one can define a ''localized version'' of   equation \eqref{eq:heat-equation} with solution $U(t,x)$, such that, whenever $x_i$ and $x_j \in \RR$ are far apart for $i\neq j$, $U(t,x_i)$, $i = 0, \pm 1, \pm 2, \dots$, are i.i.d. random variables, and also $U(t,x)$ and $u(t,x)$ are close in certain sense. Since a rare event (high peak in this case) will happen with high probability if there are enough independent random variables, i.e., $U(t,x_i)$, $i = 0, \pm1, \pm2, \dots$, one can see that $u(t,x)$, which is close to $U(t,x)$, also develops high peaks. 

Following this idea, the spatial integral $\int_R^R u(t,x) dx$ is similar to a sum of i.i.d. random variables and we expect that certain central limit theorem holds in this case. To be more precise, our first result is the following quantitative central limit theorem:
\begin{theorem}
\label{thm:TV-distance}
Suppose that $u(t,x)$ is the mild solution to equation (\ref{eq:heat-equation}) and let $F_R(t)$ be given by (\ref{eq:quantity-of-interest}). Suppose that $\sigma_R>0$. 
Let $d_{TV}$ denote the total variation distance and let $Z\sim \mathcal{N}(0,1)$. Then there exists a constant $C$, depending only on $t$, such that 
$$
d_{TV}\left( F_R(t),Z\right) \leq \frac{C}{\sqrt{R}}.
$$
\end{theorem}

\noindent
{\it Remark.}  Condition  $\sigma(1)\not=0$ guarantees that $\sigma_R >0$. Notice that this condition is not necessary. Taking into account that $\sigma_R=0$ implies
$$
\int_0^t \int_{\R} \E( \sigma^2(u(s,y))) \left(\int_{-R}^R p_{t-s} (x-y) dx \right)^2 dy ds = 0,
$$
a  sufficient condition would be that $\sigma(u(s,y))$ is not identically zero on $[0,t] \times \R$ with positive probability.

\medskip
We will show (see Proposition (\ref{pro:covariance})) that the variance $\sigma_R^2$ satisfies
\[
\lim_{R\rightarrow \infty} \frac{\sigma_R^2}{R} =   2 \int_0^t \xi(s) ds,
\]
where $\xi(s)= \E[  \sigma(u(s,y))^2]$. It turns out that $ \E[  \sigma(u(s,y))^2]$  does not depend on $y\in \R$ and is bounded on compact intervals. Then, we also prove the following functional version of  Theorem \ref{thm:TV-distance} with a normalization by $1/\sqrt{R}$.
 
\begin{theorem}
\label{thm:functional-CLT}
Suppose that $u(t,x)$ is the mild solution to equation (\ref{eq:heat-equation}).
Set $\xi(s)= \E[  \sigma(u(s,y))^2] $, $s\ge 0$.
Then, for any $T>0$,
$$
\left( \frac{1}{\sqrt{R}}\left(\int_{-R}^R u(t,x)dx - 2R\right) \right)_{t\in [0,T]} \rightarrow \left(\int_0^t \sqrt{2\xi(s)}dB_s\right)_{t\in [0,T]},
$$
as $R$ tends to infinity, where $B$ is a Brownian motion and the convergence is in law on the space of continuous functions $C([0,T])$.
\end{theorem}

Theorem  \ref{thm:TV-distance} is proved using a combination of Stein's method for normal approximations and Malliavin calculus, following the ideas introduced by Nourdin and Peccati in \cite{NP}. An innovative aspect of our methodology is to use the representation of $F_R(t)$ as a divergence, taking into account that the It\^o-Walsh integral is a particular case of the Skorohod integral.

The rest of the paper is organized as follows. In Section \ref{sec:prel} we recall some preliminaries on   Malliavin calculus and Stein's method. Sections \ref{sec:thm1} and \ref{sec:thm2} are devoted to the proofs of our main theorems. We put one technical lemma into the appendix.

\section{Preliminaries}
\label{sec:prel}

Let us first introduce the white noise on $\R_+\times\R$. 
We denote by $\mathcal{B}_b(\R_+\times \R)$ the collection of   Borel sets  $A\subset  \R_+\times\R$  with finite Lebesgue measure, denoted by $|A|$. 
Consider a centered Gaussian family of random variables
  $W= \{W(A), A\in \mathcal{B}_b\}$,  defined  in a  complete probability space $(\Omega, \mathcal{F}, P)$,
with covariance
\[
\E\left[W(A)W(B)\right] = |A\cap B|.
\]
For any $t\ge 0$, we denote  by  ${\mathcal F}_t$ the $\sigma$-field generated by  the random variables
$\{ W([0,s] \times A):0\le s\le
t,A\in {\mathcal B}_b (\R) \}$. 
  As proved in \cite{Walsh},
for any adapted random field $\left\{X(s,y),\; (s,y)\in\R_+\times\R\right\}$
that is jointly measurable and
\begin{equation}  \label{inte}
\int_0^\infty \int_\R \E[X(s,y)^2] d y d s <\infty,
\end{equation}
the following stochastic integral
\[
\int_0^\infty \int_{\R} X(s,y) W(d s, d y)
\]
is well-defined.

The proof of the main theorems relies  on  Malliavin calculus and Stein's method. Next we will   introduce the basic elements of these methodologies.  

 \subsection{Malliavin calculus} 

In this subsection we recall some basic facts on the  Malliavin calculus associated with $W$. We refer to \cite{Nualart} for a
detailed account on the Malliavin calculus with respect to a Gaussian process.
Consider the Hilbert space $\HH= L^2(\R_+\times \RR)$. The Wiener integral
\[
W(h)= \int_0^\infty \int_{\R} h(t,x) W(dt,dx)
\]
provides an isometry between  the Hilbert space $\HH$ and $L^2(\Omega)$.  In this sense $\{W(h), h\in \HH\}$ is an isonormal Gaussian process.

Denote by $C_p^{\infty}(\RR^n)$ the space of smooth functions with all their partial derivatives having at most polynomial growth at infinity. Let $\mathcal{S}$ be the space of simple random variables of the form 
\begin{equation*}
F = f(W(h_1), \dots, W(h_n))
\end{equation*}
for $f\in C_p^{\infty}(\RR^n)$ and $h_i \in \HH$, $1\leq i \leq n$. Then $DF$ is the $\HH$-valued random variable defined by
\begin{equation}
DF=\sum_{i=1}^n  \frac {\partial f} {\partial x_i} (W(h_1), \dots, W(h_n)) h_i\,.
\end{equation}
 The derivative operator $D$  is a closable operator from $L^p(\Omega)$ into $L^p(\Omega;  \HH)$ for any $p \geq1$. For any $p \ge 1$, let $\mathbb{D}^{1,p}$ be the completion of $\mathcal{S}$ with respect to the norm
\begin{equation*}
\|F\|_{1,p} = \left(\E |F|^p +   \E(  \|D F\|^p_\HH    \right)^{1/p}\,.
\end{equation*}
We denote by $\delta$ the adjoint of the derivative operator given by the duality formula
\begin{equation}\label{eq: duality formula}
\E (\delta(u) F) = \E( \langle u, DF \rangle_\HH)
\end{equation}
for any $F \in \mathbb{D}^{1,2}$, and any $u\in L^2(\Omega; \HH)$ in the domain of $\delta$, denoted by ${\rm Dom} \, \delta$. The operator $\delta$ is
also called the Skorohod integral because in the case of the Brownian motion, it coincides
with an extension of the It\^o integral introduced by Skorohod (see \cite{GT, NuPa}). 
More generally, in the context of the space-time white noise $W$, any   adapted random field $X$ which is jointly measurable and satisfies (\ref{inte}) belongs to the domain of $\delta$ and $\delta (X)$ coincides with the Walsh integral:
\[
\delta (X) = 
\int_0^\infty \int_{\R} X(s,y) W(d s, d y).
\]
As a consequence,  the mild   equation   \eqref{eq: mild} can  also be written as 
\begin{equation}\label{eq: mild Skorohod}
u(t,x) = 1 + \delta\left(p_{t-\cdot}(x-*)u(\cdot, *)\right).
\end{equation}

It is known   that for any  $(t,x)$ the solution $u(t,x)$ of equation  (\ref{eq:heat-equation}) belongs to $\mathbb{D}^{1,p}$ for any $p\ge 2$ and the derivative satisfies the following linear  stochastic integral differential equation for $t\ge s$,
\begin{eqnarray}  \notag
D_{s,y}u(t,x) &=& p_{t-s} (x-y)  \sigma(u(s,y)) \\  \label{ecu1}
&& + \int_s^t \int_{\R}     p_{t-r}(x-z)  \Sigma(r,z) D_{s,y} u(r,z) W(dr,dz),
\end{eqnarray}
where $\Sigma(r,z)$ is an adapted process, bounded by the Lipschitz constant of $\sigma$. If $\sigma $ is continuously differentiable, then
$\Sigma(r,z)= \sigma'(u(r,z))$. This result is proved in  Proposition 2.4.4  of  \cite{Nualart}  in the case of Dirichlet boundary conditions on $[0,1]$ and the proof can be easily extended to equations on $\R$. We also refer to   \cite{CHN,NQ}  for additional references where this result is used  when $\sigma$ is continuously differentiable.

\subsection{Stein's method}
Stein's method is a probabilistic technique which allows one to measure the distance between a probability distribution and normal distribution.  The total variance distance between two  random variables $F$ and $G$ is defined by 
\begin{equation}
d_{TV}(F,G) := \sup_{B \in \mathcal{B}(\RR)} |P(F \in B) - P(G \in B)|\,,
\end{equation}
where $\mathcal{B}(\RR)$ is the collection of all Borel sets in $\RR$.
We point out that  $d_{TV}(F,G)$ only depends on the laws of $F$ and $G$ and it defines a metric on the set of probability measures on $\R$. 

The following theorem provides an upper bound for the total variation distance between any  random variable and a random variable with standard normal distribution. 
\begin{theorem}\label{thm:Stein}
For $Z \sim \mathcal{N}(0,1)$ and for any  random variable $F$, 
\begin{equation}
d_{TV}(F, Z) \leq \sup_{f \in \mathscr{F}_{TV} } |\E[ f'(F)] - \E [F f(F)]| \,,
\end{equation}
where   $\mathscr{F}_{TV}$ is the class of continuously differentiable functions $f$ such that
 $\|f\|_{\infty} \leq \sqrt{\pi/2} $ and $ \|f'\|_{\infty} \leq 2 $.
\end{theorem}
 See \cite{NP} for a proof of this theorem.
Theorem  \ref{thm:Stein} can be combined with Malliavin calculus to get  the following estimate.

\begin{proposition}\label{lem: dist}
Let $F=\delta (v)$ for some $\HH$-valued random variable $v$ which belongs to ${\rm Dom }\, \delta$. Assume $\E [F^2] = 1$
and $F \in \mathbb{D}^{1,2}$. Let $Z \sim \mathcal{N}(0,1)$.  Then we have 
\begin{equation}
d_{TV}(F, Z) \leq 2 \sqrt{{\rm Var}\langle DF, v\rangle_{\HH}}\,.
\end{equation}
\end{proposition}
\begin{proof}
By our assumption on $F$, we have 
\begin{align*}
\E [F f(F)] &= \E [\delta(v) f(F)] = \E \langle v, D[f(F)]\rangle_{\HH} \\
& = \E \langle v, f'(F)DF\rangle_{\HH)} = \E \left(f'(F) \langle v, DF\rangle_{\HH}\right)\,.
\end{align*}
Thus, by Theorem \ref{thm:Stein}, 
\begin{align*}
d_{TV}(F, Z) \leq & \sup_ { f\in  \mathscr{F}_{TV}} \left|\E [f'(F)- F f(F)]\right|\\
= & \sup_ { f\in \mathscr{F}_{TV}} \left |\E [f'(F)(1-\langle DF, v \rangle_{\HH})]\right|\\
\leq &2 \E( |1-\langle DF, v \rangle_{\HH}|)\\
\leq & 2 \sqrt{\text{Var}\langle DF, v\rangle_{\HH}}\,,
\end{align*}
where the last step follows from Cauchy-Schwarz inequality, \eqref{eq: duality formula} and 
\begin{align*}
\E( \langle DF, v\rangle_{\HH} )= \E [F \delta(v)] = \E (F^2) = 1\,.
\end{align*} 
\end{proof}

In proving Theorem \ref{thm:functional-CLT} we also need the following  proposition, which is a generalization of Theorem 6.1.2 in \cite{NP}. 

\begin{proposition}\label{lemma: NP 6.1.2}
Let $F=( F^{(1)}, \dots, F^{(m)})$ be a random vector such that $F^{(i)} = \delta (v^{(i)})$ for $v^{(i)} \in {\rm Dom}\, \delta$, $i = 1,\dots, m$.  Assume $F^{(i)} \in \mathbb{D}^{1,2}$ for $i=1,\dots,m$. Let $Z$ be an $m$-dimensional Gaussian centered vector with covariance matrix $(C_{i,j}) _{1\le i,j \le m} $. For any  $C^2$ function $h: \R^m \rightarrow \R$ with bounded second partial derivatives, we have
\[
| \E h(F_R) -\E h(Z) | \le \frac 12 \|h ''\|_\infty \sqrt{   \sum_{i,j=1}^m   \E \left[ (C_{i,j} - \langle DF^{(i)}, v^{(j)} \rangle_{\HH})^2
\right] },
\]
where 
\[
\|h'' \| _\infty= \max_{1\le i,j \le m} \sup_{x\in   \R^m}  \left| \frac { \partial ^2h } {\partial x_i \partial x_j} (x) \right|.
\]
\end{proposition}
\begin{proof}
The proof will follow the same ideas as those in the proof of Theorem 6.1.2 in \cite{NP}. Without loss of generality, we may assume that $Z$ and $F$ are independent. Let 
\begin{equation*}
\Phi(t) = \E \left[h\left(\sqrt{1-t} F + \sqrt{t}Z\right) \right]\,.
\end{equation*}
Then
\begin{equation*}
\E [h(Z)]- \E [h(F) ] = \Phi(1) - \Phi(0) = \int_0^1 \Phi'(t)dt\,,
\end{equation*}
with 
\begin{equation}\label{eq: lemma 6.1.2 1}
\Phi'(t) = \sum_{i=1}^m \E \left(\frac{\partial h}{\partial x_i} \left( \sqrt{1-t} F + \sqrt{t} Z \right) \left[ \frac{1}{2\sqrt{t}} Z^{(i)} - \frac{1}{2} \frac{1}{\sqrt{1-t}} F^{(i)}\right] \right)\,.
\end{equation}
The above expression is a sum of two expectations. For the first expectation, the proof of Theorem 6.1.2 in \cite{NP} already yields that 
\begin{equation}\label{eq: lemma 6.1.2 2}
\E \left( \frac{\partial h}{\partial x_i} \left( \sqrt{1-t} F + \sqrt{t} Z \right) Z^{(i)}\right) = \sqrt{t} \sum_{j=1}^m C_{i,j} \E \left( \frac{\partial^2 h}{\partial x_i \partial x_j} \left( \sqrt{1-t} F + \sqrt{t} Z \right)\right)\,.
\end{equation}
For the second expectation, let $\E_F$ be the expectation conditioned on $Z$, then we have
\begin{align*}
&\E \left(\frac{\partial h}{\partial x_i} \left(\sqrt{1-t}F + \sqrt{t} Z \right) F^{(i)}\right)\\
=& \E \E_F \left(\frac{\partial h}{\partial x_i}\left( \sqrt{1-t} F + \sqrt{t}Z \right) \delta (v^{(i)})\right) \\
=& \E \E _F  \left(\left\langle D \frac{\partial h}{\partial x_i}\left( \sqrt{1-t} F + \sqrt{t}Z \right) ,  v^{(i)} \right \rangle_{\HH}  \right)\\
= & \sqrt{1-t}  \sum_{j=1}^m\E \left( \frac{\partial^2 h}{\partial x_i \partial x_j} \left( \sqrt{1-t} F + \sqrt{t} Z \right) \langle D F^{(j)}, v^{(i)}\rangle_{\HH}\right)\,.
\end{align*}
Finally, combining the above calculation with \eqref{eq: lemma 6.1.2 1} and \eqref{eq: lemma 6.1.2 2} with an application of Cauchy-Schwarz inequality completes the proof. 
\end{proof}

\section{Proof of Theorem \ref{thm:TV-distance}}
\label{sec:thm1}
We begin by computing the asymptotic covariance of $F_R(t)$ as $R$ tends to infinity.  This will be also relevant in the proof of Theorem \ref{thm:functional-CLT}.
\begin{proposition}
\label{pro:covariance}
Denote $\xi(r) = \E [\sigma(u(r,x))^2]$
and set
$$
G_R(t) = \int_{-R}^Ru(t,x)dx - 2R.
$$
Then, for any $s,t \ge 0$, 
$$
\lim_{R\to\infty}\frac{1}{R}{ \rm Cov}(G_R(t),G_R(s)) = 2\int_0^{s\wedge t } \xi(r)dr.
$$
\end{proposition}
\begin{proof}
Thanks to the It\^o isometry we have 
\begin{equation*}
\begin{split}
\E [u(t,x)u(s,x') ]&= 1 + \int_0^{s\wedge t }\int_\R  p_{t-r}(x-y)p_{s-r}(x'-y)\E [\sigma(u(r,y))^2]   dy dr\\
&= 1 + \int_0^{s\wedge t } \int_\R \xi(r) p_{t-r}(x'-y)p_{s-r}(x-y)dy dr  \\
&= 1 + \int_0^{s\wedge t } \xi(r) p_{t+s-2r}(x-x')dr,
\end{split}
\end{equation*}
where in the last line we have used  the semigroup property
\begin{equation}
\label{eq:heat_combination}
\int_\R p_{t}(x'-y)p_{s}(y-x)dy = p_{t+s}(x'-x).
\end{equation}
Since
$$
\E  \left(\int_{-R}^R u(t,x)dx \right) = 2R,
$$
we obtain
\begin{equation*}
\begin{split}
{\rm Cov}(G_R(t),G_R(s)) &=\int_{-R}^R \int_{-R}^R \int_0^{s\wedge t } \xi(r) p_{t+s-2r}(x-x')drdx dx' \\
&= 2\int_0^{s\wedge t } \xi(r)\int_0^{2R} p_{t+s-2r}(z)(2R-z)dz dr.
\end{split}
\end{equation*}
As a consequence,
\begin{equation*}
\begin{split}
\lim_{R\rightarrow \infty} \frac 1 R {\rm Cov}(G_R(t),G_R(s)) & =
\lim_{R\rightarrow \infty} 2 \int_0^{s\wedge t } \xi(r)\int_0^{2R} p_{t+s-2r}(z)(2-\frac zR)dz dr \\
&= 2 \int_0^{s\wedge t }  \xi(r) dr.
\end{split}
\end{equation*}
This concludes the proof. 
\end{proof}




We are now ready to prove Theorem \ref{thm:TV-distance}.
\begin{proof}[Proof of Theorem \ref{thm:TV-distance}]
By Proposition \ref{lem: dist}, we know that for any 
$F \in \mathbb{D}^{1,2} $ such that $\E (F^2)=1$ and $F=\delta(v)$,  
$$
d_{TV}(F,Z)\leq 2\sqrt{{\rm Var}(\langle DF,v\rangle_{\HH} )},
$$
where $v$ is such that $F=\delta(v)$.  
Recall that in our case we have, applying Fubini's theorem, 
\begin{equation*}
\begin{split}
F_R(t) &= \frac{1}{\sigma_R}\left(\int_{-R}^R u(t,x)dx - 2R\right) \\
&= \frac{1}{\sigma_R}\left(\int_{-R}^R \int_0^t \int_\R p_{t-s}(x-y)\sigma(u(s,y))W(ds,dy)dx\right) \\
&=\int_0^t \int_\R  \left(\frac{1}{\sigma_R}\int_{-R}^R p_{t-s}(x-y)\sigma(u(s,y))dx \right) W(ds,dy).
\end{split}
\end{equation*}
As a consequence,  taking into account equation   \eqref{eq: mild Skorohod}, we have, for any fixed $t\ge 0$, $F_R(t) =\delta(v_R)$, where
$$
v_R(s,y) =\mathbf{1} _{[0,t]}(s) \frac{1}{\sigma_R}\int_{-R}^R p_{t-s}(x-y)\sigma(u(s,y))dx.
$$
Moreover,
$$
D_{s,y}F_R =\mathbf{1} _{[0,t]}(s) \frac{1}{\sigma_R}\int_{-R}^RD_{s,y}u(t,x)dx.
$$
Therefore,
\[
\langle DF_R(t),v_R\rangle_{\HH}  = \frac{1}{\sigma^2_R}\int_0^t\int_\R\int_{-R}^R\int_{-R}^R p_{t-s}(x-y)\sigma(u(s,y))D_{s,y}u(t,x') dxd x'dyds.
\]
From (\ref{ecu1}), we know that
\begin{equation}\label{eq: D sy u}
\begin{split}
D_{s,y}u(t,x') =& p_{t-s} (x'-y)  \sigma(u(s,y)) \\
& + \int_s^t \int_{\R}     p_{t-r}(x'-z) \Sigma(r,z) D_{s,y} u(r,z) W(dr,dz)\,,
\end{split}
\end{equation}
where $\Sigma(r,z)$ is a bounded and adapted random field.
This produces the decomposition
\begin{equation}\label{eq: DF,v}
\begin{split}
\langle DF_R(t),v_R\rangle_{\HH}  =& 
 \frac{1}{\sigma^2_R}\int_0^t\int_\R  \left( \int_{-R}^R p_{t-s}(x-y)  dx \right)^2\sigma^2(u(s,y))dyds \\
 &+ \frac{1}{\sigma^2_R}\int_0^t\int_\R\int_{-R}^R\int_{-R}^R p_{t-s}(x-y)\sigma(u(s,y)) \\
 &\quad \times \left(
 \int_s^t \int_{\R}     p_{t-r}( \tilde{x} -z) \Sigma(r,z) D_{s,y} u(r,z) W(dr,dz) \right)d x d \tilde{x}dyds.
 \end{split}
 \end{equation}
 Therefore, using that   for any process $\Phi= \{\Phi(s), s\in [0,t]\}$ such that  $ \sqrt{{\rm Var} (\Phi_s)} $ is integrable on $[0,t]$,  we have
 \[
  \sqrt{ {\rm Var}  \left( \int_0^t \Phi_s ds \right)} \le \int_0^t  \sqrt{{\rm Var} (\Phi_s)} ds,
  \]
  we can write
 \[
 \sqrt{ {\rm Var} (\langle DF_R(t),v_R\rangle_{\HH} ) } \le A_1+A_2,
 \]
 where  
 \begin{eqnarray*}
&&  A_1=\frac{1}{\sigma^2_R} \int_0^t   \Bigg(   \int_{\R^2}   \left( \int_{-R}^R p_{t-s}(x-y)  dx \right)^2 \left( \int_{-R}^R p_{t-s}(x'-y' )  dx' \right)^2 \\
 && \qquad \times    {\rm Cov} \left(\sigma^2(u(s,y)),\sigma^2(u(s,y')) \right)dy dy'  \Bigg)^{\frac 12} ds
  \end{eqnarray*}
  and
   \begin{eqnarray*}
  A_2 &=&  \frac{1}{\sigma^2_R} \int_0^t \Bigg(   \int_{\R^2} \int_{[-R,R]^4}   p_{t-s}(x-y)p_{t-s}(x'-y' )     \int_s^t \int_{\R}  p_{t-r}(\tilde{x}-z)p_{t-r}(\tilde{x}'-z)  \\
&&   \times \E\left(
\sigma(u(s,y))\sigma(u(s,y'))\Sigma^2(r,z)   D_{s,y} u(r,z) D_{s,y'} u(r,z)\right) \\
&&\times  dz  dr dxdx'd\tilde{x}d\tilde{x}'dy  d{y}' \Bigg)^{\frac 12} ds \,.
 \end{eqnarray*}
 The proof will be done in two steps:
 
 \medskip
 \noindent {\it Step 1:} \quad 
 Let us first estimate the term $A_2$. Denote by $L$ the  Lipschitz constant of $\sigma$ and let, for $p\ge 2$, 
 \[
 K_p(t) =\sup_{0\le s\le t} \sup_{y\in \R} \| \sigma(u(s,y)) \|_p.
 \]
 Then, 
    \begin{eqnarray*}
&&  |\E(
\sigma(u(s,y))\sigma(u(s,y')) \Sigma^2(r,z)   D_{s,y} u(r,z) D_{s,y'} u(r,z))| \\
&& \qquad \le K_4^2(t) L^2 \| D_{s,y} u(r,z) \|_4 \| D_{s, y'} u(r,z) \|_4.
 \end{eqnarray*}
 We need to estimate  $\| D_{s,y} u(r,z) \|_p$ for any $p\ge 2$. According to \eqref{eq: D sy u}, for any $s \in [0,r]$, applying Burkholder's inequality yields
       \begin{align*}
     &   \| D_{s,y} u(r,z) \|_p  \le    p_{r-s}(z-y) K_p(t) \\
      &  \qquad +  c_p\left( E \left( \left |  \int_s^r \int_{\R} p^2_{r-r_1} (z-z_1)    \Sigma^2(r_1,z_1) |  D_{s,y} u(r_1,z_1)|^2 dr_1 dz_1 \right|^\frac p2\right) \right)^{\frac 1p}\\
    & \qquad \qquad    \le  p_{r-s}(z-y) K_p(t) \\
       &  \qquad  \qquad \ \ \   + L c_p\left(  \int_s^r \int_{\R} p^2_{r-r_1} (z-z_1)   \| D_{s,y} u(r_1,z_1) \|_p^2 dr_1 dz_1   \right)^{\frac 12},
      \end{align*}   
      which implies 
       \begin{eqnarray*}
     &&   \| D_{s,y} u(r,z) \|_p^2   \le  2p^2_{r-s}(z-y) K^2_p(t) \\
         &&  \qquad + 2L ^2c_p^2  \int_s^r \int_{\R} p^2_{r-r_1} (z-z_1)   \| D_{s,y} u(r_1,z_1) \|_p^2 dz_1 dr_1\,.
      \end{eqnarray*}  
By Lemma \ref{lemma: iteration}, we have  the estimate
      \begin{equation}
        \| D_{s,y} u(r,z) \|_p \le C  p_{r-s}(z-y)\,,  \label{e1}
        \end{equation}
        where the constant $C$ depends on $t$ and $p$.

     From (\ref{e1})  and Proposition \ref{pro:covariance}, we  derive the following estimate for the term $A_2$:
        \begin{eqnarray*}
  A_2 &\le &\frac{C}{R } \int_0^t \Bigg(   \int_{\R^2} \int_{[-R,R]^4}     \int_s^t  \int_{\R}   p_{t-s}(x-y)p_{t-s}(x'-y' )  p_{t-r}(\tilde{x}-z)     \\
&&  \times   p_{t-r}(\tilde{x}'-z)     p_{r-s}(z-y)p_{r-s}(z-y')   dz   dr   dxdx'd\tilde{x} d\tilde{x}'    dy dy'    \Bigg)^{\frac 12}ds \,.
 \end{eqnarray*}
 Integrating $\tilde{x}, \tilde{x}'$ over $\R$, then integrating $y', y$ over $\R$ and using the semigroup property,  we obtain
         \begin{eqnarray*}
 A_2 &\le & \frac{C}{R } \int_0^t \left(     \int_{[-R,R]^2}     \int_s^t   \int_{\R}  p_{t+r-2s} (x-z) p_{t+r-2s}(x'-z)   dz   dr     dxdx'     \right)^{\frac 12}ds\\
 &\le & \frac{C}{R } \int_0^t \left(     \int_{[-R,R]^2}     \int_s^t    p_{2t+2r-4s} (x -x')      dr     dxdx'     \right)^{\frac 12}ds.
  \end{eqnarray*} 
 Finally, integrating 
  $x$ over $\R$ and   $x'$ over $[-R, R]$, we get 
 \begin{equation*}
 A_2 \leq \frac{C}{\sqrt{R}}\,.
 \end{equation*}

\medskip
 \noindent {\it Step 2:} \quad 
 To estimate the term $A_1$ we need a bound for the covariance 
 \[
  {\rm Cov} \left(\sigma^2(u(s,y)),\sigma^2(u(s,y')) \right)\,.
  \]
  Here, the main idea is to use  a version of Clark-Ocone formula for two-parameter processes to write
  \[
  \sigma^2(u(s,y)) = \E[\sigma^2(u(s,y))] +  \int_0^s \int_{\R} \E[ D_{r,z} (\sigma^2(u(s,y))) | \mathcal{F}_r ] W(dr,dz).
  \]
  Then,
       \begin{eqnarray*}
        &&   {\rm Cov} (\sigma^2(u(s,y)),\sigma^2(u(s,y')) ) \\
        &&\quad =   \int_0^s \int_{\R} \E \left[ \E[ D_{r,z} (\sigma^2(u(s,y))) | \mathcal{F}_r ] \E[ D_{r,z} (\sigma^2(u(s,y'))) | \mathcal{F}_r ]  \right] dzdr .
                 \end{eqnarray*}
                 Applying  the chain rule for Lipschitz functions (see \cite[Proposition 1.2.4]{Nualart}), we have
                 \[
                 D_{r,z} (\sigma^2(u(s,y)))  =2 \sigma(u(s,y)) \Sigma(s,y) D_{r,z} u(s,y).
                 \]
                 and
                 \[
              \left\|  \E[ D_{r,z} (\sigma^2(u(s,y))) | \mathcal{F}_r ]    \right\| _2 \le 2K_4(t) L  \left\| D_{r,z} u(s,y)\right\|_4.
              \]
                    Then, using (\ref{e1}), we can write
        \begin{eqnarray*}
        &&   \left|{\rm Cov} \left(\sigma^2(u(s,y)),\sigma^2(u(s,y')) \right)\right| \\
        &&\quad \le   4L^2 K_4^2(t)    \int_0^s \int_{\R}   \left\| D_{r,z} u(s,y)\right\|_4 \left\| D_{r,z} u(s,y')\right\|_4 dzdr \\
        && \quad \le C  \int_0^s \int_{\R}    p_{s-r}(z-y) p_{s-r}(z-y') dzdr\\
          && \quad = C  \int_0^s     p_{2s-2r}(y-y') dr  .      
                 \end{eqnarray*}
                 Therefore,
                  \begin{eqnarray*}
&&  A_1 \le \frac C R  \int_0^t  \Bigg(   \int_{\R^2}   \left( \int_{-R}^R p_{t-s}(x-y)  dx \right)^2 \left( \int_{-R}^R p_{t-s}(x'-y' )  dx' \right)^2 \\
 && \qquad \times     \int_0^s     p_{2s-2r}(y-y') dr  dy dy'  \Bigg)^{\frac 12}  ds\\
 &&  \le \frac C R  \int_0^t   \Bigg(   \int_0^s  \int_{\R^2}   \int_{ [-R,R] ^4}   p_{t-s}(x-y)   p_{t-s}(\tilde{x} -y)  p_{t-s}(x' -y' ) p_{t-s}(\tilde{x}'  -y' ) \\
 && \qquad \times  p_{2s-2r}(y-y')  dx d\tilde{x} dx' d\tilde{x}'   dy dy'  dr  \Bigg)^{\frac 12}ds.
  \end{eqnarray*}
  Again, integrate $\tilde{x}$ and $\tilde{x}'$ over $\R$, then integrate $y$ and $y'$    over $\R$ using the semigroup property, to obtain
\[
  A_1 \le   \frac C R  \int_0^t   \left(   \int_0^s     \int_{ [-R,R] ^2}     p_{2t-2r}  (x-x') dxdx' dr    \right)^{\frac 12}ds.
\]
  Finally,  integrating $x$ over $\R$ and $x'$ from $-R$ to $R$, we obtain  
    \[
  A_1 \le  \frac C {\sqrt{R}}.
  \] 
 This completes the proof of Theorem  \ref{thm:TV-distance}. 
 \end{proof}

\section{Proof of Theorem \ref{thm:functional-CLT}}
\label{sec:thm2}
We begin with the following result that ensures tightness. 

\begin{proposition}
\label{pro:tightness}
Let $u(t,x)$ be the solution to equation \eqref{eq:heat-equation}. Then for any $0\leq s < t\leq T$ and any $p\geq 1$ there exists a constant $C=C(p,T)$ such that 
$$
\E \left(\left|\int_{-R}^R u(t,x)dx - \int_{-R}^Ru(s,x)dx\right|^p \right) \leq CR^{\frac{p}{2}}(t-s)^{\frac{p}{2}}.
$$
\end{proposition}
 
\begin{proof}
Let us assume that $s < t$. Recall that 
$$
u(t,x) = 1+ \int_0^t \int_\R  p_{t-r}(x-y)\sigma(u(r,y))W(dr,dy), 
$$
and thus
\begin{align*}
&\int_{-R}^Ru(t,x)dx - \int_{-R}^Ru(s,x)dx\\
&\qquad  =
\int_0^T  \int_\R \int_{-R}^R dx  \left(p_{t-r}(x-y)\1_{\{r\leq t\}} - p_{s-r}(x-y)\1_{\{r\leq s\}}\right)\sigma(u(r,y))W(dr,dy).
\end{align*}
Moreover, recall that $\E ( | u(s,y)|^p)$ is bounded on $s\leq T$ and $y\in \R$  for any $p\geq 1$. 
Using Burkholder-Davis-Gundy inequality, we can write
\begin{equation*}
\begin{split}
&\E\left(\left|\int_{-R}^R u(t,x)dx - \int_{-R}^Ru(s,x)dx\right|^p \right)\\
\leq &c_p\E \left( \int_0^T \int_\R   \left(\int_{-R}^R \left(p_{t-r}(x-y)\1_{\{r\leq t\}} - p_{s-r}(x-y)\1_{\{r\leq s\}}\right)dx \right)^2 \sigma(u(r,y))^2 dy dr\right)^{\frac{p}{2}}\\
\leq&c_p \left(\int_0^T \int_\R  \left(\int_{-R}^R \left(p_{t-r}(x-y)\1_{\{r\leq t\}} - p_{s-r}(x-y)\1_{\{r\leq s\}}\right)dx\right)^2 \|\sigma(u(r,y))\|_p^2 dy dr \right)^{\frac{p}{2}}\\
\leq & C_{p,T}\left(\int_0^T \int_\R  \left(\int_{-R}^R \left(p_{t-r}(x-y)\1_{\{r\leq t\}} - p_{s-r}(x-y)\1_{\{r\leq s\}}\right)dx\right)^2 dy dr \right)^{\frac{p}{2}}\,.
\end{split}
\end{equation*}
Thus it suffices to prove  that
\begin{equation}
\label{eq:needed}
\int_0^T \int_\R  \left(\int_{-R}^R \left(p_{t-r}(x-y)\1_{\{r\leq t\}} - p_{s-r}(x-y)\1_{\{r\leq s\}}\right)dx\right)^2 dy dr \leq CR(t-s)\,.
\end{equation}
Using Fourier transform we have
\begin{align*}
&\int_0^T \int_\R  \left(\int_{-R}^R \left(p_{t-r}(x-y)\1_{\{r\leq t\}} - p_{s-r}(x-y)\1_{\{r\leq s\}}\right)dx\right)^2 dy dr\\
=&C \int_0^T \int_{\RR} \left( \int_{-R}^R e^{i\xi x} dx \right)^2  \left(e^{-\frac{t-r}{2}|\xi|^2} \1_{\{r\leq t\}} - e^{-\frac{s-r}{2}|\xi|^2} \1_{\{r\leq s\}} \right)^2d\xi dr\\
=&C \int_0^T \int_{\RR} \frac{\sin^2(R |\xi|)}{|\xi|^2}  \left(e^{-\frac{t-r}{2}|\xi|^2} \1_{\{r\leq t\}} - e^{-\frac{s-r}{2}|\xi|^2} \1_{\{r\leq s\}} \right)^2d\xi dr\\
= & C \int_0^s \int_{\RR} \left(e^{-\frac{t-r}{2}|\xi|^2} - e^{-\frac{s-r}{2}|\xi|^2} \right)^2 \frac{\sin^2(R|\xi|)}{|\xi|^2} d\xi dr\\
& + C \int_s^t \int_{\RR} e^{-(t-r)|\xi|^2} \frac{\sin^2(R|\xi|)}{|\xi|^2} d\xi dr\\
:=& C(I_1+ I_2)\,.
\end{align*}
For $I_1$ we can write
\begin{align*}
I_1 &= \int_0^s \int_{\RR} e^{r|\xi|^2} \left(e^{-\frac{t}{2}|\xi|^2} - e^{-\frac{s}{2}|\xi|^2} \right)^2 \frac{\sin^2(R|\xi|)}{|\xi|^2} d\xi dr \\
& = \int_{\RR} \frac{1-e^{-s|\xi|^2}}{|\xi|^2} \left(e^{-\frac{t-s}{2}|\xi|^2} -1 \right)^2 \frac{\sin^2(R|\xi|)}{|\xi|^2} d\xi \,.
\end{align*}
Using the bound $|1-e^{-a}|\leq \sqrt{a}$ for all $a\geq 0$ in the above parenthesis, we obtain that 
\begin{align*}
I_1\leq \int_{\RR} \frac{1}{|\xi|^2} |\xi|^2 |t-s| \frac{\sin^2(R|\xi|)}{|\xi|^2} d\xi  \leq C R |t-s|\,.
\end{align*}
For $I_2$, using the bound $1-e^{-a}\leq a$ for any $a\geq 0$, 
\begin{align*}
I_2 = \int_{\RR} \frac{1-e^{-(t-s)|\xi|^2}}{|\xi|^2} \frac{\sin^2(R|\xi|)}{|\xi|^2} d\xi \leq (t-s)\int_{\RR}\frac{\sin^2(R|\xi|)}{|\xi|^2} d\xi = C R (t-s)\,.
\end{align*}
The proof is finished by combining $I_1$ and $I_2$. 
\end{proof}

\begin{proof}[Proof of Theorem \ref{thm:functional-CLT}]
It suffices to prove the convergence of the finite-dimensional distributions and tightness.  However, the latter follows directly from Proposition \ref{pro:tightness}.

In order to show the convergence of the finite-dimensional distributions, fix points $0\le t_1< \cdots <t_m \le T$ and consider the random variables
\[
F_R^{(i)}  = \frac 1 {\sqrt{R}} \left(\int_{-R}^R u(t_i,x)dx - 2R\right),
\]
for $i=1, \dots, m$.  We can write $F_R^{(i)}  = \delta( v_R^{(i)}) $, where
$$
v_R^{(i)} (s,y) =\mathbf{1} _{[0,t_i]}(s) \frac{1}{\sqrt{R} }\int_{-R}^R p_{t_i-s}(x-y)\sigma(u(s,y))dx.
$$
Set $F_R=( F^{(1)}_R, \dots, F_R^{(m)})$ and let $Z$ be an $m$-dimensional Gaussian centered vector with covariance
\[
C_{i,j} :=\E[Z^iZ^j]  = \int _0^{t_i \wedge t_j}  \xi(r)  dr,
\]
where we recall that $\xi(r)= \E[ \sigma(u(r,x))^2]$.
Then, applying Proposition \ref{lemma: NP 6.1.2},
for any  $C^2$ function $h: \R^m \rightarrow \R$ with bounded second partial derivatives, we have
\[
| \E(h(F_R)) -\E(h(Z)) | \le \frac 12 \|h ''\|_\infty \sqrt{   \sum_{i,j=1}^m   \E \left[ \left(C_{i,j} - \langle DF_R^{(i)}, v_R^{(j)} \rangle_{\HH}\right)^2 \right] }.
\]
Then, it suffices to show that for each $i,j$,   $\langle DF_R^{(i)}, v_R^{(j)} \rangle_{\HH}$ converges in $L^2$, as $R$ tends to infinity to  $C_{i,j} $. To be more precisely, similarly with \eqref{eq: DF,v}, we have 
\begin{equation}\label{eq: DFi,vj}
\begin{split}
& \langle DF_R^{(i)},v_R^{(j)}\rangle_{\HH}\\
  =& 
 \frac{1} {R}  \int_0^{t_i\wedge t_j}\int_\R   \int_{-R}^R\int_{-R}^R p_{t_i-s}(x-y)   p_{t_j-s}(\tilde{x}-y)  \sigma^2(u(s,y))dx d\tilde{x}dyds \\
 &+ \frac{1} {R} \int_0^{t_i\wedge t_j}\int_\R\int_{-R}^R\int_{-R}^R p_{t_j-s}(\tilde{x}-y)\sigma(u(s,y)) \\
 &\quad \times \left(
 \int_s^{t_i} \int_{\R}     p_{t_i-r}(x-z) \Sigma(r,z)  D_{s,y} u(r,z) W(dr,dz) \right)dxd \tilde{x}dyds\\
:= & I_{1,i,j}(R) + I_{2, i, j}(R) \,.
 \end{split}
 \end{equation}
Then we obtain that 
\begin{align*}
\E \left[\left(C_{ij} - \langle D F_{R}^{(i)}, v_R^{(j)}\rangle_{\HH}\right)^2\right]
\leq 2\E \left(C_{ij} - I_{1,i,j}(R) \right)^2 + 2\E \left(I_{2,i,j}(R)^2\right)\,.
\end{align*}
By noting that 
\begin{equation*}
C_{ij} = \lim_{R \to \infty} \E( I_{1,i,j}(R))\,,
\end{equation*}
and using   arguments similar  as  those in the proof of Theorem \ref{thm:TV-distance}, we can show that $\E \left[\left(C_{ij} - \langle D F_{R}^{(i)}, v_R^{(j)}\rangle_{\HH}\right)^2\right] \to 0$ as $R \to \infty$. The proof is finished. 

\medskip

\end{proof}
  
 \section{Appendix}
 Here we prove a technical lemma which is used in the proof of Theorem \ref{thm:TV-distance}. 
 \begin{lemma}\label{lemma: iteration}
 Let $\|D_{s,y}u(r,z)\|_p^2$ satisfy 
 \begin{equation}\label{eq: Du satisfy}
  \| D_{s,y} u(r,z) \|_p^2   \le  Cp^2_{r-s}(z-y)   + C   \int_s^r \int_{\R} p^2_{r-r_1} (z-z_1)   \| D_{s,y} u(r_1,z_1) \|_p^2 dz_1 dr_1 \,,
 \end{equation}
 for any $0 < s < r\leq t$ and $y, z \in \R$. Then we have
 \begin{equation*}
  \| D_{s,y} u(r,z) \|_p \le C_{t,p}  p_{r-s}(z-y)\,.
 \end{equation*}
 for some constant $C_{t,p}$ which depends on $t$ and $p$. 
  \end{lemma}
  \begin{proof}
  By iterating \eqref{eq: Du satisfy} we have 
  
 \begin{align*}
& \|D_{s,y}u(r,z)\|_p^2\\
 \leq & C p_{r-s}^2(z-y) + C^2 \int_s^r\int_{\R} p_{r-r_1}^2(z-z_1) p_{r_1-s}^2(z_1-y) dz_1 dr_1\\
 & + \cdots  + C^n \int_s^r\int_{\R} \int_{s}^{r_1} \int_{\R}\cdots \int_s^{r_{n-1}} \int_{\R} p_{r-r_1}^2(z-z_1)p_{r_1-r_2}^2(z_1-z_2)  \times  \cdots \\
 &\quad  \times p_{r_{n-1} - r_{n}}^2(z_{n-1}-z_n) p_{r_n-s}^2(z_n-y)  dz_n dr_n \cdots dz_1 dr_1  \\
& +  \int_s^r\int_{\R} \int_{s}^{r_1} \int_{\R}\cdots \int_s^{r_{n}} \int_{\R} p_{r-r_1}^2(z-z_1)p_{r_1-r_2}^2(z_1-z_2)\times  \cdots \\
 &\quad  \times p_{r_{n-1} - r_{n}}^2(z_{n-1}-z_n) p_{r_n-r_{n+1}}^2(z_n-z_{n+1})\\
 & \quad \times \|D_{s,y}u(r_{n+1}, z_{n+1})\|_p^2 dz_{n+1} dr_{n+1}  \cdots dz_1 dr_1\,.
 \end{align*}
Then, using the fact that 
\[
p^2_{r-s} (z)   =\frac {C}{\sqrt{  r-s}}  p_{\frac {r-s}2} (z) \,,
\] 
and 
$$
\int_{\R} p_s(x-y)p_t(y-z) dy = p_{s+t}(x-z)\,,
$$
we have for the integrals of the product of heat kernels above, 
\begin{align*}
&\int_s^r\int_{\R} \int_{s}^{r_1} \int_{\R}\cdots \int_s^{r_{n-1}} \int_{\R} p_{r-r_1}^2(z-z_1)p_{r_1-r_2}^2(z_1-z_2)  \times  \cdots \\
 &\quad  \times p_{r_{n-1} - r_{n}}^2(z_{n-1}-z_n) p_{r_n-s}^2(z_n-y)  dz_n dr_n \cdots dz_1 dr_1\\
=& C^n \int_s^r \int_s^{r_1}\cdots \int_s^{r_{n-1}} \int_{\R^n} \frac{1}{\sqrt{r-r_1}}\frac{1}{\sqrt{r_1-r_2}} \cdots \frac{1}{\sqrt{r_{n-1}-r_n}} \frac{1}{\sqrt{r_n-s}}\\
& \times p_{\frac{r-r_1}{2}}(z-z_1)p_{\frac{r_1-r_2}{2}}(z_1-z_2)  \times \cdots  \times p_{\frac{{r_{n-1} - r_{n}}}{2}}(z_{n-1}-z_n) p_{\frac{r_n-s}{2}}(z_n-y) \\
&\times dz_n\cdots dz_1 dr_n \cdots dr_1\\
= & C^n \int_s^r \int_s^{r_1}\cdots \int_s^{r_{n-1}} \frac{1}{\sqrt{r-r_1}}\frac{1}{\sqrt{r_1-r_2}} \cdots \frac{1}{\sqrt{r_{n-1}-r_n}} \frac{1}{\sqrt{r_n-s}} p_{\frac{r-s}{2}}(z-y) dr_n \cdots dr_1\\
= & (r-s)^{\frac{n-1}{2}} \frac{\Gamma(\frac{1}{2})^{n+1}}{\Gamma(\frac{n+1}{2})} p_{\frac{r-s}{2}}(z-y)\\
= & (r-s)^{\frac{n}{2}}  \frac{\Gamma(\frac{1}{2})^{n+1}}{\Gamma(\frac{n+1}{2})} p_{r-s}^2 (z-y)\,.
\end{align*}
Thus we obtain 
\begin{align*}
&\|D_{s,y}u(r,z)\|_p^2 \leq \left( \sum_{j=0}^n C^j(r-s)^{\frac{j}{2}}\frac{\Gamma(\frac{1}{2})^{j+1}}{\Gamma(\frac{j+1}{2})}\right) p^2_{r-s}(z-y)\\
& \quad + C^n (r-s)^{\frac{n}{2}} \frac{\Gamma(\frac{1}{2})^{n+1}}{\Gamma(\frac{n+1}{2})} \\
& \quad \times \int_s^r \int_{\R}  p^2_{r-r_{n+1}}(z-z_{n+1}) \|D_{s,y}u(r_{n+1}, z_{n+1})\|_p^2 dz_{n+1} dr_{n+1}\,.
\end{align*}
Taking into account that the first term on the right-hand side is a convergent series and the second term tends to 0 as $n \to \infty$, the proof is finished. 
  \end{proof}

\end{document}